\documentclass[11pt]{amsart}

\usepackage{amsmath, amssymb, latexsym}

\usepackage{amsthm}

\newtheorem{theorem}{Theorem}[section]
\newtheorem{lemma}[theorem]{Lemma}
\newtheorem{proposition}[theorem]{Proposition}
\newtheorem{corollary}[theorem]{Corollary}

\theoremstyle{definition}
\newtheorem{definition}[theorem]{Definition}
\theoremstyle{remark}
\newtheorem{remark}[theorem]{Remark}
\theoremstyle{remark}
\newtheorem{example}[theorem]{Example}
\theoremstyle{remark}

\theoremstyle{remark}

 \newcounter{saveenumi}

\begin{document}

\title{On Algebraic Properties of Topological Full Groups}
\author{R. Grigorchuk and K. Medynets}

\date{}

\address{Rostislav Grigorchuk,  Department of Mathematics,
Texas A\&M University, College Station, TX 77843-3368, USA}
\email{grigorch@math.tamu.edu}

\address{Konstantin Medynets, Ohio State University, Department of Mathematics, Columbus, OH  43210-1174, USA}
\email{medynets@math.ohio-state.edu}

\begin{abstract} In the  paper we discuss  the algebraic structure of topological full group $[[T]]$ of a Cantor minimal system $(X,T)$. We show that  the topological full group $[[T]]$  has the structure  similar to a union of permutational wreath products of group $\mathbb Z$. This  allows us to prove that the topological full groups  are  locally embeddable into finite groups; give an elmentary proof of the fact that group $[[T]]'$ is infinitely presented; and provide explicit examples of maximal locally finite subgroups of $[[T]]$.   We also show that the commutator subgroup $[[T]]'$, which is simple and finitely-generated for minimal subshifts, is decomposable into a product of two locally finite groups and that the groups $[[T]]$  and $[[T]]'$ possess continuous ergodic  invariant random subgroups. 
\end{abstract}

\maketitle

\section{Introduction}

In the paper we study algebraic properties of a completely new, from the  geometric group theory point of view, class of groups --- the full groups of dynamical systems. Our goals are  to develop the machinery for the study of full groups,  establish a number of new results, and to survey known facts. We also hope to attract attention of specialists in group theory to  full groups as they  possess   very unusual algebraic properties, which seem to have never appeared in the literature  before. As yet another motivation to the study of full groups, we would  like to mention a recent preprint \cite{Thomas:2012}, where full groups were applied to the complexity problem of isomorphism relation between finitely-generated simple groups.   Since the theory of full groups lies on the intersection of the theory of dynamical systems and group theory, we  take an extra care while discussing dynamical results  so that the paper becomes accessible to non-specialists in dynamics.

The full groups first appeared 50 years ago in the seminal paper of H.~Dye \cite{Dye:1963} as an algebraic invariant of  orbit equivalence for measure-preserving dynamical systems. Even though, the full groups contain {\it complete information} about orbits of the underlying systems,  their algebraic structure  has remained largely unknown. 
 Let $G$ be a group acting on a set $X$. The set $X$ may be considered along with some structure such as  a measure, a Borel $\sigma$-algebra, or a topology. Consider the group $[G]$, termed the {\it full group} of $G$, consisting of all automorphisms $S$ of $X$ preserving the  structure such that for every $x\in X$, $S(x) = g(x)$ for some $g\in G$. It turns out that in many cases the group $[G]$ is so ``rich'' that any isomorphism $\alpha$ between $[G]$ and $[H]$, where a group $H$ also acts on $X$ by structure-preserving automorphisms, will always be spatially generated in the sense that there is an automorphism $\Lambda : X\rightarrow X$ with $\alpha(g) = \Lambda \circ g\circ \Lambda^{-1}$ for every $g\in [G]$.  Such  results have been established for groups acting on a Cantor set   \cite[Theorem 4.5(c)]{Rubin:1989},  \cite{GlasnerWeiss:1995}, \cite{GiordanoPutnamSkau:1999},   \cite{BezuglyiKwiatkowski:2000},   \cite{BezuglyiMedynets:2008},  and \cite{Medynets:2011};  on non-compact zero-dimensional spaces \cite{Matui:2002}; standard Borel spaces
  \cite{Miller_Rosendal:2007}; and measure spaces   \cite{Dye:1963} and \cite{Zhuravlev:PhD}. As an application to the theory of dynamical systems, these results  show that algebraic (group) properties of full groups  completely determine the class of orbit equivalence for underlying  dynamical systems. This brings in a question of  interaction between  dynamical characteristics of the system and algebraic properties of full groups.

  In the present paper, we  focus on algebraic properties of full groups for   systems acting on Cantor sets. Consider a pair $(X,T)$, where $X$ is a Cantor set and $T: X\rightarrow X$ is a homeomorphism of $X$.  Denote by $[T]$ the full group of $\{T^n\}_{n\in \mathbb Z}$.  Consider a subgroup $[[T]]$, termed the {\it topological full group},  consisting of all homeomorphisms  $S$ of the set $X$ such that  $Sx=T^{f_S(x)}x$ for every point $x$, where $f_S: X\rightarrow \mathbb Z$ is a continuous  function.  {\it We would like to emphasize that no group topology is involved and the word ``topological'' refers to the setting of topological dynamics. }  The group $[[T]]$ is countable and its commutator subgroup is simple and finitely generated under some assumptions on $(X,T)$ (listed further in the text).

One of the main techniques in Cantor dynamics is the method of periodic approximation, which mimics the behaviour of $(X,T)$ by periodic transformations.  This roughly  means that there exists a refining sequence of clopen (Kakutani-Rokhlin) partitions $\{\Xi_n\}_{n\geq 1}$ (Section \ref{SectionPreliminaries}) such that the action of $T$ can be seen as  nearly a permutation of  $\Xi_n$-atoms. We will use the Kakutani-Rokhlin tower analysis to  show that every element of the topological full group $[[T]]$ can be uniquely represented as a product of a  permutation from $Sym(\Xi_n\textrm{-atoms})$ and   induced transformations of $T$ on some clopen sets (Theorem \ref{TheoremRepresentation}). The proof of this result indicates that the topological full group $[[T]]$ has structure similar to  an increasing union  of wreath products  $\mathbb Z$ Wr $Sym(\Xi_n\textrm{-atoms})$. We then use Theorem \ref{TheoremRepresentation} to establish that {\it the topological full group of a Cantor minimal system is locally embeddable into finite groups} (LEF groups) (Theorem \ref{IntroMainResult}).

As corollaries of  Theorem \ref{IntroMainResult}, we  obtain the following results: (1) topological full groups are sofic; (2)  topological full groups are not finitely presented (established originally by H.~Matui \cite{Matui:2006}) (here we use arguments from \cite{grigorchuk:1984} related to the convergence in the space of marked groups);  and (3) the universal theory of topological full groups coincide with that of finite groups. We will also prove that the commutator subgroup $[[T]]'$, while being simple and finitely generated for minimal subshifts, can be represented as $[[T]]' = \Gamma_1\Gamma_2$, where $\Gamma_i$, $i=1,2$ is a (maximal) locally finite subgroup. To the best of our knowledge, this is the first example of simple finitely generated groups with such a factorization property. We will also construct explicit examples of maximal locally finite subgroups.

The structure of the paper is the following.
In Section \ref{SectionAlgebricProperties}, we list known   and establish several new algebraic properties of topological full groups.

Section \ref{SectionPreliminaries} is devoted to basics of Cantor dynamics. We introduce all necessary definitions from the theory of dynamical systems and explain the essence of  Kakutani-Rokhlin tower analysis, which will be the main technical ingredient in our proofs. We would like to mention that the present paper is perhaps the first work that uses  Kakutani-Rokhlin partitions  for the study of algebraic properties of full groups.

In Section \ref{SectionRotationsPermutations} we prove  Theorem \ref{TheoremRepresentation}.  This result was inspired by the paper of Bezuglyi and Kwiatkowski \cite{BezuglyiKwiatkowski:2000}.

Section \ref{SectionLEF} is devoted to the proof of the LEF property and to the discussion of locally finite subgroups.

\bigskip
{\bf Acknowledgement.} We would like to thank  Hiroki Matui, Sergey Bezuglyi,  Jan Kwiatkowski  for numerous helpful discussions of full groups. We are also thankful to  Mathieu Carette, Dennis Dreesen,  Mark Sapir, and Vitaliy Sushchanskyy for their  comments.

%
%

\section{Algebraic Properties of Topological Full Groups}\label{SectionAlgebricProperties}

This section can be deemed as a short survey on algebraic properties of full groups as we list results  scattered over numerous papers. Throughout the paper the symbol $X$ will stand for a Cantor set, i.e. for any zero-dimensional compact metric space without isolated points, and  $T: X\rightarrow X$ will denote a homeomorphism of $X$.  Recall that all Cantor sets are homeomorphic to each other and, in particular, to the space of sequences $\{0,1\}^\mathbb N$.

Consider a pair $(X,T)$, where $X$ is a Cantor set and $T: X\rightarrow X$ is a homeomorphism.  The pair $(X,T)$ is called a {\it Cantor dynamical
 system}. If  $X$ has no proper closed $T$-invariant  subsets, then $T$ is called {\it minimal}.   We will always assume that a dynamical system $(X,T)$ is {\it minimal}, i.e., if $Y\subset X$ is a non-empty closed set with $T(Y)=Y$, then $Y = X$. The minimality of $T$ is equivalent to the property that every $T$-orbit $\{T^n(x) : n\in\mathbb Z\}$ is dense  in $X$. Observe that if $T$ is a minimal homeomorphism, then $T^k$, $k\neq 0$, might be non-minimal, though, it will still have no finite orbits. A homeomorphism that has no finite orbits is called {\it aperiodic}. The following is an equivalent definition of the topological full group.

\begin{definition}\label{DefinitionFullGroups} A {\it topological full group} $[[T]]$ is defined as a group consisting of all homeomorphisms $S$ for which there is a finite clopen partition $\{C_1,\ldots,C_m\}$ of $X$ and a set of integers $\{n_1,\ldots,n_m\}$ such that $S|C_i = T^{n_i}|C_i$ for every $i=1,\ldots, m$.
\end{definition}
 
 Consider the topological full group $[[T]]$.  For a point $x\in X$, denote by $[[T]]_x$  the set of all elements $S\in [[T]]$ with $S(\{T^n(x) : n\geq 0\}) = (\{T^n(x) : n\geq 0\})$.  It follows from the compactness and zero-dimensionality of the Cantor set that the group $[[T]]$ is countable. The groups  $[[T]]$ and $[[T]]'$ (the commutator subgroup) are complete algebraic  invariants for  flip conjugacy of the system $(X,T)$,  \cite{GiordanoPutnamSkau:1999} and \cite{BezuglyiMedynets:2008}.   We recall that two dynamical systems $(X,T)$ and $(Y,S)$ are {\it flip conjugated} if either $T$ and $S$ or $T$ and $S^{-1}$ are topologically conjugated.    This implies that dynamical invariants of  flip conjugacy such as the topological entropy \cite[Theorem 7.3]{Walters:book}, spectral characteristics,  the number of invariant ergodic measures, and others, can be used to distinguish  full groups and their commutator subgroups up to  isomorphism.
 Observe that for every positive real number $\alpha$, there exists a minimal subshift with the topological entropy equal to $\alpha$, see  \cite{HahnKatznelson:1967} and \cite{Grillenberger}. This implies that there are uncountably many non-isomorphic topological full groups as well as their commutator subgroups. In particular, {  \it there are uncountably many non-isomorphic finitely generated simple  groups  satisfying the properties listed  in the following theorem   and other results of the present paper}.

\begin{theorem}\label{TheoremPropertiesFullGroups}

(1)  For every $S\in [[T]]$, let $\varphi(S) = \int_X f_S(x) d\mu(x)$, where $\mu$ is a probability $T$-invariant measure.  Then $\varphi$   is a homomorphism from $[[T]]$ onto $\mathbb Z$ \cite[Section 5]{GiordanoPutnamSkau:1999}. Thus, the group $[[T]]$ is indicable.

(2) For every point $x\in X$, the group $[[T]]_x$ is locally finite \cite{GiordanoPutnamSkau:1999}.

(3) If $x,y\in X$ are points lying in different  $T$-orbits, then $Ker(\varphi) = [[T]]_x\cdot [[T]]_y$ is a product of two locally finite groups \cite[Lemma 4.1]{Matui:2006}.

(4)  Any finite group can be embedded into $[[T]]'$, Remark \ref{RemarkInfiniteAbelianGroup}. The group $[[T]]'$ contains a subgroup isomorphic to an infinite direct sum of the group $\mathbb Z$, Remark \ref{RemarkInfiniteAbelianGroup}.

(5) The commutator subgroup $[[T]]'$ is simple \cite[Theorem 4.9]{Matui:2006}. Furthermore, if $H$ is a normal subgroup in $[[T]]$, then $[[T]]'\subset H$  \cite[Theorem 3.4]{BezuglyiMedynets:2008}.

(6) The commutator subgroup $[[T]]'$ is finitely generated if and only if $(X,T)$ is topologically isomorphic to a minimal subshift over a finite alphabet \cite[Theorem 5.4]{Matui:2006}.

(7) If $(X,T)$ is a minimal subshift, then the group $[[T]]'$ is not finitely presented \cite[Theorem 5.7]{Matui:2006}.

(8) The commutator subgroup $[[T]]'$ contains the lamplighter subgroup if and only if $(X,T)$ is not conjugated to an odometer (a rotation on $\{p_n\}$-adic integers).
\end{theorem}

 H.~Matui \cite{Matui:2006} was apparently the first author who obtained purely algebraic results in the theory of  full groups of Cantor systems. We notice, though, that the existence of   connections between  properties of full groups and those of underlying dynamical systems could already be seen from Dye's type results on orbit equivalence, see \cite{GiordanoPutnamSkau:1999} and the  aforementioned references.

In the first version of the present paper, we conjectured that the topological full group for a minimal Cantor system is amenable.  K.~Juschenko and N.~Monod \cite{JuschenkoMonod} have recently announced an affirmative solution of our conjecture.

\begin{theorem}[\cite{JuschenkoMonod}] The topological full group of a Cantor minimal system is amenable.
\end{theorem}

We  notice that the minimality (or, at least, aperiodicity)  of the dynamical system seems to be crucial for the amenability of topological full groups as, for example, the topological full group of any shift of finite type  contains a free non-abelian subgroup [H.~Matui, private communications].
 We also mention that there are examples of {\it minimal} Cantor $\mathbb Z^2$-actions, whose topological  full groups contain free non-abelian subgroups \cite{Elek_Monod}.

Denote by $AG$ the class of amenable groups and by $SAG$, termed the {\it class of subexponentially amenable groups}, the class of groups containing groups of subexponential growth  and closed under  (I) direct unions and (II) group extensions, see  \cite{grigorchuk:1998} for details. The family $SAG$ can be described by a transfinite induction. Namely, denote by $SAG_0$ the class of groups of subexponential growth. If $\alpha$ is not a limit ordinal, set $SAG_\alpha$ to be the class of groups obtained from $SAG_{\alpha-1}$ either by operation (I) or by operation (II). If $\alpha$ is a limit ordinal, set $SAG_\alpha = \bigcup_{\beta<\alpha}SAG_\beta$.  Then $SAG = \bigcup_\alpha SAG_\alpha$, where $\alpha$ runs over all ordinals, see \cite{Chou:1980} for a related construction of elementary amenable groups.

 \begin{proposition} Let $(X,T)$ be a minimal subshift over a finite alphabet. Then none of the groups $[[T]]$ or $[[T]]'$ is  subexponentially amenable.
 \end{proposition}
 \begin{proof} It is enough to show that $[[T]]'$ is not subexponentially amenable as the class $SAG$ is closed under passing to subgroups. Assume the converse. Choose the least ordinal $\alpha\geq 0$ with $[[T]]'\in SAG_\alpha$.  Since the lamplighter group has exponential growth, we get that $[[T]]'\notin SA_0$ (Theorem \ref{TheoremPropertiesFullGroups}(8)). Thus, $\alpha>0$.

  As $[[T]]'$ is finitely generated (Theorem \ref{TheoremPropertiesFullGroups}), $\alpha$ cannot be a limit ordinal. So, $[[T]]'$ has to be an extension of a group from $SAG_{\alpha-1}$ by a group from the same class. This is impossible in view of the simplicity of the group $[[T]]'$. This contradiction yields the result.
 \end{proof}

\begin{definition}  A group $G$ is called {\it locally embeddable into finite groups} (abbr. LEF) if for every finite set $F\subset G$ there is a finite group $H$ and a map $\varphi : G\rightarrow H$ such that (a) $\varphi$ is injective on $F$ and (2) $\varphi(gh)=\varphi(g)\varphi(h)$ for every $g,h\in F$.
\end{definition}

The notion of an LEF group was introduced by Vershik and Stepin in 1980's as a group property equivalent to the existence of  uniform free  approximations for   group actions.  We refer the reader to  \cite{VershikGordon:1997} and \cite[Chapter 7]{CeccheriniCoornaert} for  a detailed exposition of the theory of LEF groups, see also references therein for complete historical  information.  The proof of the following result is presented in Section \ref{SectionLEF}.

\begin{theorem}\label{IntroMainResult} The topological full group of any Cantor minimal system is an LEF group.
\end{theorem}

We mention that  the LEF property and amenability do not imply each other  \cite{VershikGordon:1997}, \cite{grigorchuk:1998}, and \cite[Chapter 7]{CeccheriniCoornaert}, yet both imply the soficity \cite[Corollary 7.5.11]{CeccheriniCoornaert}. A  group is called {\it sofic} if it embeds into an ultraproduct of a sequence of finite symmetric groups endowed with a normalized  Hamming metric. 

\begin{corollary}  The topological full group $[[T]]$ and its commutator subgroup $[[T]]'$ are sofic.
\end{corollary}
We would  like to mention that G.~Elek and N.~Monod suggested a different approach related to some ideas from \cite{JuschenkoMonod}  to obtain the LEF property for the topological full group of {\it minimal  $\mathbb Z$-subshifts} [private communications].

 \begin{remark}  Since there are uncountably many non-isomorphic finitely generated groups $[[T]]'$, we get that  the class $AG\setminus SAG$  contains continuum of simple finitely generated LEF groups. 
 \end{remark}

Each finitely generated LEF group  can be also defined as a limit of finite groups in the topological space of marked groups (or Cayley graphs) introduced in \cite{grigorchuk:1984}.  If a finitely presented group  $G$ is a limit of a sequence of finite groups $\{G_n\}$, then starting from some index $n$ all groups $G_n$ are quotients of $G$.  Thus, finitely presented infinite simple groups are isolated points in the space of marked groups \cite[Section 2]{grigorchuk:2005} and, therefore, cannot have the LEF\footnote{This type of arguments  was used by the first author  to show that all groups of intermediate growth constructed in \cite{grigorchuk:1984} are infinitely presented.} property. Hence, Theorem \ref{IntroMainResult} along with simplicity of $[[T]]'$ imply that  $[[T]]'$, when $T$ is a minimal subshift over a finite alphabet, is a finitely generated, but {\it infinitely presented} group. The latter result was first obtained by H.~Matui by different (dynamical) methods.

\begin{corollary}   For every  minimal subshift $T$ over a finite alphabet, the commutator subgroup $[[T]]'$ is infinitely presented.
\end{corollary}

In section \ref{SectionLEF}, we will modify the proof of Lemma 4.1 from \cite{Matui:2006} to establish that  the commutator subgroup   can be represented as a product of two locally-finite subgroups.  We notice that groups admitting factorizations $G = A\cdot B$ with $A$ and $B$ being locally finite subgroups have been studied in numerous papers, see, for example, publications of Amberg, Chernikov, Kegel,  Subbotin, Sushchanskyy,  Sysak, and many others, see, for example, \cite{Sushchanskii:1990} and references therein.  To the best of our knowledge, commutators of topological full groups are the first examples of {\it finitely-generated simple groups} that can be factorized into a product of two locally-finite subgroups.
For a point $x\in X$, set $\Gamma_x = [[T]]_x\cap [[T]]'$. Note that each group $\Gamma_x$ is locally finite \cite{GiordanoPutnamSkau:1999}.

\begin{theorem} Let $x,y\in X$ be points with disjoint $T$-orbits. Then $[[T]]' = \Gamma_x\cdot \Gamma_y$.
\end{theorem}

\begin{definition} (1) A group $G$ is said to {\it satisfy a group law} if there exists $k\geq 1$ and a word $w\in \mathbb F_k$, free group of rank $k$, such that $w(g_1,\ldots,g_k) = 1$ for any $g_1,\ldots,g_k\in G$.

(2) If $G$ acts on a set  $Y$, $G$ is said to {\it separate} $Y$ if for any finite set $C\subset Y$, the pointwise stabilizer $st_G(C)$ does not stabilize any other point outside $C$.
\end{definition}

 In \cite{Abert:2005} M.~Abert shows that if a permutation group $G$ separates $Y$, then $G$ does not satisfy any group law. We notice that the topological full group $[[T]]$ separates $X$ since for any clopen set $O$ and a point $x\subset O$ there is $\gamma \in [[T]]'$ with $supp(g)\subset O$ and $g(x)\neq x$, see, for example, \cite[Proposition 2.4]{Medynets:2011}. Thus, we immediately get the following result.

 \begin{proposition} The topological full group $[[T]]$  satisfies no group law.
 \end{proposition}

\bigskip  Consider the space $\textrm{SUB}_T$ of all subgroups of $[[T]]$ (identified with a closed subspace of $\{0,1\}^{[[T]]}$ via characteristic functions).   Recall that any group acts on the set of its subgroups by conjugation. This action will be referred to as {\it adjoint}. 

  \begin{definition} A probability   measure on the space of all subgroups invariant with respect to the adjoint action is called an {\it invariant random subgroup}.
  \end{definition}

An action of a group $G$ on the set $X$ is called {\it completely non-free} if for any two distinct points $x,y\in X$, their pointwise stabilizers $st_G(x)$ and $st_G(y)$ are different subgroups of $G$. We note that the actions of $[[T]]$ and $[[T]]'$ on $X$ are completely non-free. 

  \begin{proposition} Let $(X,T)$ be a Cantor minimal system. Then both groups $[[T]]$ and $[[T]]'$ have  nonatomic (ergodic) invariant random subgroups. 
  \end{proposition}
  \begin{proof}   Since distinct points have different stabilizers, the map $\alpha : X\rightarrow \textrm{SUB}_T$ given by $\alpha(x) = st_{[[T]]}(x)$ is injective. We notice that any element $g \in st_{[[T]]}(x)$ also stabilizes some neighborhood of $x$ as the action of the group $\{T^n\}$ on $X$  is free. Hence    the map $\alpha : X\rightarrow \textrm{SUB}_T$  is a homeomorphism on its image, see  \cite[Lemma 5.4]{Vorobets:arxiv}. This implies that the dynamical systems $(\alpha(X),[[T]])$ and $(X,[[T]])$ are conjugate. So any $T$-invariant measure on $X$ gives rise to a random subgroup on $[[T]]$. The proof for the group $[[T]]'$ is verbatim. A $T$-ergodic measure will produce an ergodic random subgroup.
  \end{proof}
  
  Since   the group $[[T]]'$ is simple, it has no atomic invariant subgroups except for $\{id\}$. The group $[[T]]'$, on the contrary, has normal subgroups $[[T]]'$ and $Ker(\varphi)$, see Theorem \ref{TheoremPropertiesFullGroups}. Thus, Dirac  measures supported by these subgroups are non-trivial invariant random subgroups of $[[T]]$.

  \medskip The following observation  was suggested by Mark Sapir. Recall that a {\it universal sentence} in a first order language $\mathcal L$ is any sentence of the form $\forall x_1,\ldots,\forall x_k \Phi$, where $\Phi$ is a quantifier free formula. The {\it universal theory} for a class of groups $\mathcal K$ is the family of all universal sentences that are valid in all groups from $\mathcal K$. The universal theory of the class $\mathcal K$ is denoted by $\textrm{Th}_\forall(\mathcal K)$. A group $G$ is called a {\it model} of a set of sentences  $W$ if every sentence from $W$ holds in $G$. We recall that a group $G$  is LEF if and if it is embeddable into an ultraproduct of finite groups \cite[Theorem 7.2.5]{CeccheriniCoornaert}. The latter is a
 model for $\textrm{Th}_\forall(\mathcal F)$, where $\mathcal F$ is the class of finite groups. 

\begin{proposition}  For any minimal Cantor system $(X,T)$   the universal theory of $[[T]]'$  coincides with  the universal theory  of the class of finite groups   and hence is undecidable.
\end{proposition}
\begin{proof}   For any Cantor minimal system $(X,T)$, the commutator subgroup $[[T]]'$ contains any finite group (Theorem \ref{TheoremPropertiesFullGroups}).  Thus, $\textrm{Th}_\forall([[T]]')\subset \textrm{Th}_\forall(\mathcal F)$. On the other hand, since the group $[[T]]'$  is LEF, it can be embedded into an ultraproduct of finite groups. Hence,  $\textrm{Th}_\forall(\mathcal F) \subset \textrm{Th}_\forall( [[T]]')$. 

The  result of Slobodskoi \cite{Slobodskoi:1981} about the undecidability of the universal theory for finite groups completes the proof. 
\end{proof}

We expect that the word problem for  $[[T]]'$ is decidable if  the system $(X,T)$ can be  defined effectively.

%
%


\section{Kakutani-Rokhlin Partitions} \label{SectionPreliminaries}

In this section we fix our notations and introduce necessary definitions from the theory of Cantor dynamical systems.  An interested reader may also consult the papers   \cite{HermanPutnamSkau:1992},  \cite{giordano_putnam_skau:1995}, \cite{GlasnerWeiss:1995}, \cite{GiordanoPutnamSkau:1999},  \cite{BezuglyiKwiatkowski:2000}, \cite{Matui:2006}, and \cite{BezuglyiMedynets:2008} for more results on Cantor dynamics.

We will sometimes need to use a {\it diameter} of  a set. By this we mean that  the diameter is defined by some  metric on $X$ compatible with the topology.
 If $d$ is a metric  on $X$ and $\mathcal U$ is an open cover of $X$, then, due to the compactness of $X$, there exists a number $\delta>0$ such that any set $A\subset X$ with $d$-diameter less than $\delta$ is contained in at least one member of the cover $\mathcal U$. The number $\delta$ is called the {\it Lebesgue number of the cover $\mathcal U$}.

\begin{definition} Consider a clopen set $B$  such that the  sets $$\xi =\{B,TB,\ldots, T^{n-1}B\}$$ are disjoint. The family $\xi$ is called a {\it $T$-tower with  base $B$ and height $n$}. A clopen partition of $X$ of the form $$\Xi = \{T^iB_v : 0\leq i\leq h_v-1, \; v=1,\ldots, q\}$$ is called a {\it Kakutani-Rokhlin partition}.  The sets $\{T^iB_v\} $
 are called {\it atoms} of the partition $\Xi$.
\end{definition}

Fix an arbitrary clopen set $A\subset X$. Define a function $t_A : A \rightarrow \mathbb N$ by setting $$t_A(x) = \min \{ k \geq 1 : T^kx\in A\}.$$ Using the minimality of $T$, one can check that half-orbits $\{T^nx : n\geq 0\}$, $x\in X$, are dense in $X$.  This shows that the function $t_A$ is well-defined and takes on finite values. Since $T$ is a homeomorphism of a Cantor set, the function  $t_A$ is continuous and, hence, is bounded.  Sometimes, $t_A$ is referred to as the {\it function of the first return.}

Denote by $K$ the set of all integers $k\in \mathbb N$ such that the set $A_k = \{x\in A : t_A(x) = k\}$ is non-empty.   It follows from the continuity of $t_A$ that the set $K$ is finite and $A= \bigsqcup _{k\in K} A_k$ is a clopen partition.  The definition of the function $t_A$ implies that, for every $k\in K$, the sets $ \{A_k,TA_k,\ldots, T^{k-1}A_k\}$ are disjoint.  Indeed if $T^iA_k\cap T^jA_k \neq \emptyset$ for $0\leq i < j\leq k-1$, then $A_k\cap T^{j - i}A_k\neq \emptyset$. It follows that there exists $x\in A_k$ with $T^{j-i}x\in A_k$. Hence $ k = t_A(x)\leq j-i\leq k-1$, which is a contradiction.  Similarly, one can check that the family $$\Xi = \{T^iA_k : 0\leq i\leq k-1,\;k\in K\}$$ consists of disjoint sets and
$$X = \bigsqcup _{k\in K}\bigsqcup_{i = 0}^{k-1}T^iA_k.$$ Thus, $\Xi$  is a Kakutani-Rokhlin partition of $X$.

\begin{definition}

The union of the sets $B(\Xi) = \bigsqcup_{k\in K}A_k$ is called the {\it base of the partition $\Xi$} and the union of the top levels $H(\Xi) = \bigsqcup_{k\in K} T^{k-1}A_k$ is called the {\it top or roof of the partition}. \end{definition}

\begin{remark}\label{RemarkKakutaniRokhlin} (1) A convenient way to look at Kakutani-Rokhlin partitions is by tracing the trajectories of points. Namely, if $x\in A_k$, for some $k\in K$, then $x$ moves to the level $TA_k$ under the action of $T$. Consequently applying $T$ to the point $x$, we see that $x$ is moving up until it gets on the  top level, i.e. $T^{k-1}x\in  H(\Xi)$. When we  apply $T$ once again,   the point $T(T^{k-1}x)$ returns back to the set $A = \bigsqcup _{i\in K} A_i$.  However, we cannot predict the exact set $A_i$, $i\in K$,    the point $T(T^{k-1}x) = T^kx$ gets into.

(2) Note that $T(H(\Xi)) = B(\Xi)$.

(3)  Let $B$ be an arbitrary clopen set. We can partition each set $A_k$, $k\in K$, into a finite number of clopen subsets $\{C_{j,k}\}$, $j=0,\ldots,p_k$, so that each set $T^iC_{j,k}$, $0\leq i\leq k-1$, is either disjoint from $B$ or is a subset of $B$. Thus, $\Xi' = \{T^iC_{j,k}: k\in K, j=0,\ldots,p_k,i =0,\ldots,k-1\}$ is a clopen partition {\it refining} the partition $\Xi$ and the set $B$.

\end{remark}

Fix a point $x_0\in X$. Choose a decreasing sequence of clopen sets $\{E_n\}_{n\geq 1}$ with $\bigcap_{n\geq 1}E_n = \{x_0\}$. Denote by $\Xi_n$ the clopen partition constructed by the function $t_{E_n}$ as above. Refining partitions $\Xi_n$ as in Remark \ref{RemarkKakutaniRokhlin}(3) if needed, we can assume that the sequence $\{\Xi_n\}_{n\geq 1}$ satisfies the following conditions.

 \begin{enumerate}

 \item The partitions $\{\Xi_n\}_{n\geq 1}$ generate the topology of $X$.

 \item The partition $\Xi_{n+1}$ refines $\Xi_n$.

 \item  $\bigcap_{n\geq 1} B(\Xi_n) = \{x_0\}$ and $B(\Xi_{n+1})\subset B(\Xi_n)$ for every $n\geq 1$.

 \setcounter{saveenumi}{\theenumi}
 \end{enumerate}

Since $\Xi_n$ is a Kakutani-Rokhlin partition, we can represent it as $$\Xi_n = \{T^iB^{(n)}_v : 0\leq i\leq h^{(n)}_v-1,  v=1,\ldots, v_n\}$$  for some clopen set $B^{(n)}_v$ and positive integers $v_n$ and $h_v^{(n)}$, $v \in V_n = \{1,\ldots, v_n\}$. Set  $\xi^{(n)}_v = \{B^{(n)}_v,\ldots, T^{h^{(n)}_v -1} B_v^{(n)}\}$.  Then $\Xi_n$  is a disjoint union of towers $\xi^{(n)}_v$, $v\in V_n$. Note that $h_v^{(n)}$ is the height of the tower $\xi^{(n)}_v$.

Fix a sequence of positive integers $\{m_n\}$ such that $m_n\to\infty$ as $n\to\infty$. Take a subsequence of   $\{\Xi_n\}_{n\geq 1}$ (we will drop an extra subindex) so that each Kakutani-Rokhlin partition $\Xi_n$ additionally meets the following conditions.

 \begin{enumerate}
   \setcounter{enumi}{\thesaveenumi}

  \item    $h_n\geq 2m_n +2$,  where $h_n = \min_{v\in V_n}h^{(n)}_v$.

 \item The sets $T^iB(\Xi_n)$ have the property $$diam(T^iB(\Xi_n))<1/n\mbox{ for }-m_n-1\leq i \leq m_n.\qquad (\ddag)$$
  \end{enumerate}

 \begin{remark} Observe that we  do not need  the minimality of $T$ to get partitions satisfying properties (1)-(4).  Such partitions exist for any aperiodic\footnote{A homeomorphism is called {\it aperiodic} if every orbit is infinite.} homeomorphism \cite{BezuglyiDooleyMedynets:2005}.  However, the condition (5) holds only for minimal  (or, at least, essentially minimal\footnote{A system is called {\it essentially minimal} if it has only one minimal component.}) systems \cite{HermanPutnamSkau:1992}.
 \end{remark}




%
%
\section{Rotations and Permutations}\label{SectionRotationsPermutations}
  In this section, we introduce two kinds of group elements:  permutations (Definition \ref{DefinitionPermutation}) and rotations (Definition \ref{DefinitionRotation}). We then show that every element can be uniquely factored into a product of a permutation and of a rotation (Theorem \ref{TheoremRepresentation}).  The results of this section are inspired by the paper \cite{BezuglyiKwiatkowski:2000}, cf.  \cite[Theorem 3.3]{Medynets:2007}.  Theorem 2.2 of  \cite{BezuglyiKwiatkowski:2000}  describes   how elements of $[[T]]$ move atoms of Kakutani-Rokhlin partitions.
Fix a sequence of partitions $\{\Xi_n\}_{n\geq 1}$ that satisfies the properties (1) -- (5) from Section \ref{SectionPreliminaries}.  We use the same notations as in Section \ref{SectionPreliminaries}.

\begin{definition}\label{DefinitionPermutation} Fix an integer $n\geq 1$.   We say that a homeomorphism $P\in [[T]]$ is an {\it $n$-permutation} if (1) its orbit cocycle $f_P$ is compatible (constant on atoms) with the partition $\Xi_n$ and (2) for any point $x\in T^i B_v^{(n)}$ ($0\leq i\leq h^{(n)}_v-1$, $v\in V_n$) we have that $0\leq f_P(x) + i \leq h_v^{(n)}-1$. The latter condition means that $P$ permutes atoms only within each tower without moving points over the top or the base of the tower. We will  call $P$ just a {\it permutation} when the partition $\Xi_n$ is clear from the context.
\end{definition}

 Recall that the set $V_n = \{1,\ldots,v_n\}$ stand for the index set  enumerating $T$-towers of $\Xi_n$. Then each permutation $P$ can be uniquely factored into a product of permutations $P_1,\ldots,P_{v_n}$ such that $P_i$ acts only within the tower $\xi_i^{(n)}$, $i=1,\ldots, v_n$.

\begin{definition} Fix a clopen set $A$. Let $t_A$ be the function of the first return to the set $A$. Define a homeomorphism $T_A$ by $T_A(x) = T^{t_A(x)}x$, when $x\in A$, and $Tx=x$ if otherwise. The homeomorphism $T_A$ belongs to $[[T]]$ and is called an {\it induced transformation of $T$}.
\end{definition}

\begin{remark}\label{RemarkInfiniteAbelianGroup} (1) Note that the minimality of $T$ implies the minimality of the induced transformation $T_A$ on $A$.

(2) Since minimal homeomorphisms have no periodic points,   the group generated by the induced transformation $T_A$ is isomorphic to $\mathbb Z$.  Choose a sequence of disjoint clopen sets $\{A_n\}_{n\geq 1}$ such that $A_n = B_n'\sqcup B_n''$ with $B_n'' = T^{q_n}(B_n)$ for some $q_n$. Then the element $Q_n = T_{B_n'}\cdot T_{B_n''}^{-1}$ belongs to $[[T]]'$ as $T_{B_n'}$ and $T_{B_n''}$ are conjugate in $[[T]]$. Hence,  the subgroup of $[[T]]'$ generated by $\{Q_n\}_{n\geq 1}$ is isomorphic to the infinite direct sum   $ \bigoplus_{i=1}^\infty\mathbb Z$.

(3)   Given an integer $n>0$, choose a clopen set $U$ such that the sets $\mathcal F = \{U,TU,
\ldots, T^{n-1}U\}$ are disjoint. Then using powers of $T$, we can embed a symmetric group $Sym(n)$ into $[[T]]$ by acting on $
\mathcal F$. In particular, any finite group is embeddable into $[[T]]$. One can modify these arguments to show that any finite group also embeds into $[[T]]'$. 
\end{remark}

For each integer $0\leq i\leq h_n-1$, set $$U(i) = \bigsqcup_{v\in V_n} T^{h^{(n)}_v-i -1}B_v^{(n)}\mbox{ and }D(i) = \bigsqcup_{v\in V_n}T^{i}B_v^{(n)}.$$

The set $U(i)$ ($D(i)$) consists of atoms that are at the distance $i$ from the top (base) of the partition $\Xi_n$.
Notice that  the induced transformation $T_{U(i)}$ has the form

$$T_{U(i)}(x) = \left\{\begin{array}{lll}x & \mbox{if}& x\notin U( i); \\
T^{h_w^{(n)}}x & \mbox{if}& x\in T^{h_v^{(n)} - i-1}B_v^{(n)}\mbox{ and }T^{i+1}x\in B_w^{(n)}.\end{array}\right.$$
Similarly, the transformation $T_{D(i)}^{-1}$ has the form
$$T_{D(i)}^{-1}(x) = \left\{\begin{array}{lll}x & \mbox{if}& x\notin D( i); \\
T^{-h_w^{(n)}}x & \mbox{if}& x\in T^{i}B_v^{(n)}\mbox{ and }T^{-i-1}x\in T^{h_w^{(n)}-1}B_w^{(n)}.\end{array}\right.$$

\begin{remark}
(1) Informally, the action of the homeomorphism $T_{U(i)}$ can be expressed as follows.  The homeomorphism $T_{U(i)}$ moves any point $x\in T^{h_v^{(n)} - i -1}B^{(n)}_v$ to the top of the tower $\xi_v^{(n)}$; then to the base of some tower $\xi_w^{(n)}$; eventually this point is moved up to the level $T^{h^{(n)}_w -1 -i}B_w^{(n)}$. This can be seen from the decomposition
$$h_w^{(n)} =  i + 1 + (h_w^{(n)} - 1 - i).$$  Here the first summand shows the number of levels to the top of the tower $\xi_v^{(n)}$; the number $1$ tells us that the point is mapped onto the base of the partition $\Xi_n$; the last summand shows the number of steps needed to move the point to the level of tower $\xi_w^{(n)}$ that is at the same distance from the top as the level $T^iB_v^{(n)}$. Observe that a similar description can be also applied to $T_{D(i)}^{-1}$.

(2) Notice that $T_{U(i)}$ and $T_{D(i)}$ belong to $[[T]]$. Observe also that their orbit cocycles might not be compatible with the partition $\Xi_n$; though the cocycles will be compatible with some other partition $\Xi_m$, $m>n$.
\end{remark}

In the following definition, the sequence of integers $\{m_n\}_{n\geq 1}$ satisfies  Equation (\S) of Section \ref{SectionPreliminaries}.

\begin{definition}\label{DefinitionRotation} A homeomorphism $R\in [[T]]$ is called an {\it $n$-rotation with the rotation number less or equal to $r>0$} if there are two sets $S_d,S_u\subset \{0,\ldots, m_n\}$ such that
$$R  = \prod_{i\in S_u}(T_{ U(i)})^{l_i} \times \prod_{ j \in S_d} (T_{D( j)})^{k_j}  $$
for some   integers    $\{l_i\}$ and $\{k_j\}$ with $|l_i|\leq r$ and $|k_j| \leq r$. The sets $S_u$ and $S_d$ will be  called {\it supportive sets} for $R$.

Since $U(i)\cap D(j)  = \emptyset$ for all $0\leq i\neq j\leq m_n < h_n/2$, the rotation $R$ is well-defined and its definition does not depend on the order of the product.
\end{definition}

 The supportive set $S_d$ shows which levels within the distance $m_n$ to the base of the partition $\Xi_n$ are ``occupied'' by the support of $R$. The set $S_u$ shows the same, but for levels within the distance $m_n$ of the top of $\Xi_n$. 
 We would like to emphasize that the the numbers $\{l_i\}$ and $\{k_j\}$ in the definition of rotations can take on arbitrary (positive and negative) integer values with absolute values not exceeding $r$.  One can construct rotations with  rotation numbers being any given integer.
If $R_1$ and $R_2$ are rotations with rotation numbers not exceeding $r_1$ and $r_2$, respectively, then the rotation number of $R_1R_2$ does not exceed $r_1+r_2$.

\begin{example}\label{ExampleOdometer} We illustrate the definitions of rotations and permutations on the example of topological full group of the 2-odometer. Furthermore,  we will completely describe the structure of this group. The description was suggested by H.~Matui in private communications.

Set $Y = \{0,1\}^{\mathbb N}$. Recall that the $2$-odometer $O: Y\rightarrow Y$  is defined as $O(0^n1w) = 1^n0w$ for any sequence $1^n0w\in Y$ and $O(1^\infty) = 0^\infty$.  For every $n\geq 1$, set $B_0^{(n)} = \{x\in Y : x_0\ldots x_{n-1} = 0^n\}$ and $B_i^{(n)} = O^i B_0^{(n)}$, $0< i < 2^n-1$.  Then
$\Xi_n = \{B_0^{(n)},\ldots, B_{2^n-1}^{(n)}\}$ is a  sequence Kakutani-Rokhlin partitions $\{\Xi_n\}_{n\geq 1}$ satisfying conditions (1) - (5) from Section  \ref{SectionPreliminaries}.  Observe that $\bigcap_{n\geq 1}B(\Xi_n) = \{0^\infty\}$.

Fix $n\geq 1$. Consider a subset $G_n\subset [[O]]$ of elements $S$ such that the orbit cocycle $f_S$ is compatible with the partition $\Xi_n$. Since the odometer $O$ cyclically permutes the atoms $\{B_0^{(n)},\ldots, B_{2^n-1}^{(n)}\}$, the element $S$ will also act as a permutation of these atoms. Denote the induced permutation on $\{0,\ldots,2^n-1\}$ by $\mathcal P_n(S)$.
 We notice  that $G_n$ is, in fact, a subgroup of $[[O]]$ and   $\mathcal P_n : G_n \rightarrow \rm{Sym}(2^n)$ is a homomorphism.

If $S\in Ker(\mathcal P_n)$, i.e., $\mathcal P_n(S) = id$, then for every $i=0,\ldots,2^n-1$ there exists $k_i\in\mathbb Z$ such that \begin{equation}\label{EquationOdometer1}S|B_i^{(n)} \equiv O_i^{k_i}|B_i^{(n)},\end{equation} where $O_i$ is the induced transformation of $O$ onto the atom $B_i^{(n)}$. Since  the partition $\Xi_n$ consists only of one tower, we get that $D(i) = U(i)$ and $O_i = O_{D(i)} = O_{U(i)}$ for every $i=0,\ldots,2^n-1$.

It follows from Equation (\ref{EquationOdometer1}) that $Ker(\mathcal P_n)$ is isomorphic to $\mathbb Z^{2^n}$.
Observe that $Ker(\mathcal P_n)$  consists exactly of $n$-rotations in our terminology. The rotation number of $S\in Ker(\mathcal P_n)$ is $$\max_{0\leq i\leq 2^n-1} |k_i|.$$

Denote by $Perm_n$ the set of all $n$-permutations of $[[O]]$. Observe that $Perm_n$ is a subgroup of $G_n$.   As $Ker(\mathcal   P_n) \cap Perm_n = \{1\}$ and $\mathcal P_n(Perm_n) = Sym(2^n)$, the group $G_n$ is a semidirect product of $Perm_n$ and $Ker(\mathcal P_n)$, which, at the same time, can be represented as a permutational wreath product of $\mathbb Z$ and $Sym(2^n)$. Notice that any element $S\in G_n$ can be uniquely factored out as $S = PR$, where $P\in Perm_n$  is an $n$-permutation and $R\in Ker(\mathcal P_n)$ is an $n$-rotation.  We will show in Theorem  \ref{TheoremRepresentation} that  a similar factorization exists in the topological full group of any Cantor minimal system.

\end{example}

 We  fix the following {\it notation}. Let $P_v$ be a permutation of levels within the tower $\xi_v^{(n)}$. Note also that $P_v$  naturally induces a permutation of the set $\{0,\ldots, h_v^{(n)}-1\}$. For  permutations $P_v$ and $P_w$ of distinct towers,  the equality $k = P_v(i) = P_w(i)$ will literally mean that the images of the atoms $T^iB_v^{(n)}$ and $T^iB_w^{(n)}$ are atoms $T^kB_v^{(n)}$ and $T^kB_w^{(n)}$.

 \begin{enumerate}

\item We will treat the set $\{0,\ldots, h_v^{(n)}-1\}$ as a cyclic group (mod) $h_v^{(n)}$ and we will denote by $[-a,b]$  the set $$[-a,b] = \{0,\ldots, b\}\cup \{ h_v^{(n)} - 1, h_v^{(n)} - 2,\ldots, h_v^{(n)} - a \}.$$ for $a,b\geq 0$.

\item We can naturally identify the group $\{0,\ldots, h_v^{(n)}-1\}$ with  $\{|z| =  1 : z^{h_v^{(n)}} = 1\}$. Denote by $d$ the metric on the unit circle normalized in such a way that  $h_v^{(n)}$ is the total length of the circle. Denote by $d_v^{(n)}$ the metric induced on the set  $\{0,\ldots, h_v^{(n)}-1\}$. Observe that $d_v^{(n)} (0,h_v^{(n)}-1) = 1$.

\end{enumerate}

The following proposition establishes that every element of $[[T]]$  can be represented as a product of permutations and rotations.

\begin{theorem}\label{TheoremRepresentation} Let $Q\in [[T]]$.  (1) Then there exists $n_0>0$ such that for all $n\geq n_0$, the homeomorphism $Q$ can be represented as $Q=PR$, where $P$ is an $n$-permutation and $R$ is an $n$-rotation with the rotation number not exceeding one.
Furthermore,  the permutation $P$ can be factorized (in a unique way) as a product of permutations $P_{1},\ldots P_{v_n}$  meeting the following conditions:

\begin{enumerate}

\item[(i)] The permutation $P_v$ acts only within the $T$-tower $\xi_v^{(n)}$, $v\in V_n$;

\item[(ii)]
  $P_v(i) = P_w(i)$ for all $i\in [-m_n,m_m]$ and  for all $v,w\in V_n$;

\item[(iii)] the map $P_v$ induces a permutation of  $\{0,\ldots h_v^{(n)} - 1\}$ (denoted by the same symbol) with  the property that $d_v^{(n)}(P_v(i),i)  \leq n_0$ for all $i\in \{0,\ldots h_v^{(n)} - 1\}$;

\item[(iv)] the rotation $R$ acts only on levels which are within the distance $n_0$  to the top or the bottom of the partition. In other words, the supportive sets of the rotation $R$ are contained in $\{0,1,\ldots,n_0-1\}$.
\end{enumerate}

(2) If  $Q= P_2 R_2$ is another factorization with $P_2$ being an $n$-permutation and $R_2$ being an $n$-rotation, then $P = P_2$ and $R = R_2$.

(3) For any finite set $F\subset [[T]]$, there exists $n_0>0$ such that for all $n\geq n_0$, the factorizations $Z = P_Z R_Z$ of elements $Z\in F$ into  $n$-permutations and  $n$-rotations satisfy $P_{Z_1}\neq P_{Z_2}$ for $Z_1,Z_2\in F$ and $Z_1\neq Z_2$.
 \end{theorem}
\begin{proof} (1) We will split the proof into several steps.

(1-a) Choose an integer $n_0$ such that the orbit cocycle $f_Q$  is compatible (constant on atoms) with the partition $\Xi_{n_0}$ and $n_0\geq \max_{x\in X}|f_Q(x)|$.  Furthermore, in view of the compactness of $X$, we can assume that $n_0$ is chosen  so large that $1/n_0$ is a Lebesgue number for the clopen partition $\{f_Q^{-1}(k) : k\in\mathbb Z\}$, i.e., any set having diameter less than $1/n_0$ is contained in the set $f_Q^{-1}(k)$ for some $k\in\mathbb Z$.
Notice that the choice of $n_0$ implies that the cocycle $f_Q$ is constant on  each set $T^iB(\Xi_{n_0})$ for $-m_n-1 \leq i \leq m_n$, see Equation (\ddag) in Section \ref{SectionPreliminaries}.

Thus,  for all $n\geq n_0$ if for some $x\in T^iB^{(n)}_v$ we have that
$f_Q(x) + i \geq h_v^{(n)}$, then $f_Q|U(h_v^{(n)} - 1- i ) \equiv const$. Since $n_0\geq \max_{z\in X}|f_Q(z)|$, we have that $h_v^{(n)} - i \leq f_Q(x) \leq  n_0$. Similarly, if $x\in T^iB^{(n)}_v$ is such that
$f_Q(x) + i < 0 $, then $f_Q|D(i) \equiv const$ and $i\leq n_0$. These properties say that if an atom $T^iB_v^{(n)}$  is moved  over the top or the bottom of $\Xi_n$ by the homeomorphism $Q$, then this atom is   taken from  the $n_0$ lowest or $n_0$ highest levels of $\Xi_n$, i.e. $d_v^{(n)}(i,0) \leq n_0$ or $d_v^{(n)}(i,h_v^{(n)}-1) \leq n_0$.

(1-b) Define a rotation $R\in [[T]]$ as follows:

$$R(x) = \left\{\begin{array}{lll}T_{U(h_v^{(n)} - 1- i)}(x) & \mbox{if} & f_Q(x) + i \geq h_v^{(n)}\\
T^{-1}_{D(i)}(x) & \mbox{if} & f_Q(x) + i < 0 \\
x &\mbox{otherwise.}\end{array}\right.$$

Observe that the rotation number of $R$ is equal to $1$ and  $R$ may act non-trivially only on the $n_0$ lowest  levels and the $n_0$ highest  levels of $\Xi_n$.

 Define a permutation $P$ as a product of $v_n$ permutations $P_1\cdots P_{v_n}$, where each permutation $P_v$ acts only within the tower $\xi^{(n)}_v$ as follows
 $$P_v|_{T^iB_v^{(n)} }= T^p|_{T^iB_v^{(n)}},\mbox{ where }p = -i + \left(i + f_Q|_{T^iB_{v}^{(n)}} \mod h_v^{(n)}\right), $$
for $i=0,\ldots, h_v^{(n)}-1$. In other words, we define the permutation $P_v$ as $Q$ if the homeomorphism $Q$ does not move  the set $T^iB_{v}^{(n)}$ over the top or the base of the tower, i.e. $0\leq f_Q(x) + i \leq h_v^{(n)}-1$ for $x\in T^iB_v^{(n)}$. If $ f_Q(x) + i \geq h_v^{(n)}$ for $x\in T^iB_v^{(n)}$, then the entire level $U(h_v^{(n)} - 1- i)$ will go over the top of $\Xi_n$ under the action of $Q$. This implies that  $QU(h_v^{(n)} - 1- i) = D(j)$ for some (unique) $j$. In this case, we require that $P_v $ sends $T^iB^{(n)}_v$ to the atom $T^jB^{(n)}_v$ under the action of $T^{j-i}$.  Now if $f_Q(x) + i <0 $ for $x\in T^iB_v^{(n)}$, then $QD(i) = U(j)$ for some $j$. We set $P = T^{h_v^{(n)}-1-j - i}$ on such a set.

(1-c) {\it We claim that $Q = PR$.} Fix $x\in X$. Find unique $v\in V_n$ and $0\leq i \leq h_v^{(n)}-1$ with $x\in T^i B_v^{(n)}$. We will consider three cases depending on the value of $ f_Q(x) + i $.

   Assume that $0\leq f_Q(x) + i \leq h_v^{(n)}-1$.
It follows from the definitions of $P$ and $R$ that $R(x) = x$ and $P(x) =  Q(x)$. Thus, $Q(x) = PR(x)$.

 Assume that $i + f_Q(x) \geq h_v^{(n)}$. Then $y := R(x) = T^{h_w^{(n)}}(x)$, where $w$ is uniquely defined by the relation $T^{h_v^{(n)} - i}x\in B^{(n)}_w$.
Note that  $h_v^{(n)} - i \leq n_0 $ and $y,x\in U(h_v^{(n)} - 1- i)$. Hence $ f_Q(y) = f_Q(x)$ by the choice of $n_0$.  Note also that $y\in T^{j}B_w^{(n)}$ with $j = h_w^{(n)}  - (h_v^{(n)} - i)$.
Hence, by the definition of $P$, we get that $P(y) = T^p(y)$ with
\begin{eqnarray*} p & =  & - j + \left(f_Q(y) + j \mod h_w^{(n)}\right) \\
& =  & - j + \left(f_Q(x) +  \left(h_w^{(n)}  - (h_v^{(n)} - i)\right) \mod h_w^{(n)}\right) \\
\\ &= &- j + \left(f_Q(x) - (h_v^{(n)} - i)\right ) \\
& = & f_Q(x) - (h_v^{(n)}-i+j) \\
& = & f_Q(x) - h_w^{(n)}.\end{eqnarray*}
It follows that $PR(x) = T^{f_Q(x) - h_w^{(n)}} T^{h^{(n)}_w}x = T^{f_Q(x)}x = Qx$.

 The case  $i + f_Q(x) <0$ can be checked similarly to the one above. We leave the details to the reader.

 (2) Assume that $PR = P_2R_2$. Hence $P_2^{-1}P = R_2 R^{-1}$. Observe that $P_2^{-1}P$ is an $n$-permutation and $R_2 R^{-1}$ is an $n$-rotation.  Notice also that if $P$ is an $n$-permutation and $R$ is an $n$-rotation, then $P\neq R$ unless they are both trivial transformations.
This immediately yields the uniqueness of the decomposition.

(3) Consider two distinct elements $Z_1$ and $Z_2$ of  $[[T]]$. Find a clopen set $C$   such that $Z_1(y)\neq Z_2(y)$ for every $y\in C$. Consider factorizations $Z_1 = P^{(1)}_nR^{(1)}_n$ and $Z_2 = P^{(2)}_nR^{(2)}_n$ into $n$-permutations and $n$-rotations. We notice that for all $n$ large enough the $n$-rotations $R^{(1)}_n$ and $R^{(2)}_n$ are supported by the sets $[-k,k]$, for some  $k>0$ independent of $n$. Hence, we can find a clopen subset $C'\subset C$ such that  for large $n$ the supports of $R^{(1)}_n$ and $R^{(2)}_n$ are disjoint from $C'$. This implies that $P^{(1)}_n|C' \neq P^{(2)}_n|C'$ for all $n$-large enough. This completes the proof.
 \end{proof}



\section{Applications}\label{SectionLEF}

In this section, we show that the topological full group of {\it any} Cantor minimal system is locally embeddable into finite groups and that its commutator is a product of two locally-finite subgroups.

\begin{theorem}\label{TheoremLEF}  Let $(X,T)$ be an arbitrary  Cantor minimal system. Then the topological full group $[[T]]$ is an LEF group.
\end{theorem}
\begin{proof} Fix a finite set $F\subset [[T]]$. We can  assume that the group identity is contained in $F$.  Set $F^2 = F\cdot F$. Using Proposition 3.9, find $n\in \mathbb N$ such that  every $Q\in F^2$ can be factored as  $Q  = S_Q R_Q$, where $R_Q$ is an $n$-rotation and $S_Q$ is an $n$-permutation  such that $S_Q \neq S_Z$ for $Q,Z\in F^2$ with $Q\neq Z$.

Since $F$ is a finite set, there is a integer $d$ such that all $n$-rotations $\{R_Q\, :\, Q\in F\}$ are supported by levels $[-d,d]$ for all $n$ large enough (Theorem \ref{TheoremRepresentation}(iv)).  Using Theorem \ref{TheoremRepresentation}(ii), we can choose $n$  so that for every $Q\in F$, we have $S_{Q,v}^{- 1}(i) = S_{Q,w}^{- 1}(i)$ for all $i\in [-d,d]$ and $v,w\in V_n$, where $S_{Q} = \prod_{v\in V_n}S_{Q,v}$. This implies that
 $S_Z^{-1}R_QS_Z$ is  an $n$-rotation for all $Z,Q\in F$.

Denote by $H$ the group of all $n$-permutations. Then $H$ is a finite group. Define a map $\varphi : F^2 \rightarrow H$ by $\varphi(Q) = \varphi(S_QR_Q) = S_Q.$ Observe that in view of Theorem \ref{TheoremRepresentation}, the factorization $Q = S_Q R_Q$ is unique for every $Q\in F^2$. Thus, the map $\varphi$  is well-defined.  Note that $\varphi$ is injective on $F$.

If $Q,Z\in F$, then
$$\varphi(QZ) = \varphi(S_QR_Q S_Z R_Z) = \varphi((S_Q S_Z) S_Z^{-1}R_Q S_Z R_Z) = S_QS_Z,$$ where the last equality follows from the uniqueness of the factorization and the fact that $S_Z^{-1}R_Q S_Z R_Z$ is an $n$-rotation and $S_Q S_Z$ is an $n$-permutation. Therefore, $\varphi(QZ) = S_Q S_Z = \varphi(Q) \varphi(Z)$. Set $\varphi(g) = id$ for all $g\notin F^2$. This shows that $[[T]]$ is an LEF group. \end{proof}

Recall that the group $[[T]]_x$ is defined as  the set of all elements $S\in [[T]]$ with $S(\{T^n(x) : n\geq 1\}) = (\{T^n(x) : n\geq 1\})$. It was observed in \cite{GiordanoPutnamSkau:1999} that the group $[[T]]_x$ is locally finite. On the language of the present paper, this can be also derived from the fact that the group $[[T]]_x$ coincides with the set of all $n$-permutations associated to a sequence $\{\mathcal P_n\}_{n\geq 1}$  of Kakutani-Rokhlin partitions converging to $\{x\}$, i.e.  $\bigcap_{n\geq 1}B(\mathcal P_n) = \{x\}$ (see Section \ref{SectionPreliminaries} for details). Indeed, if $P$ is an $n$-permutation consistent with $\mathcal P_n$,  then  $P\in [[T]]_x$  as $P$ does not move points over the top or bottom of the partition. Conversely, if $Q \in [[T]]_x$, then $Q = PR$ is a product of $n$-permutation and $n$-rotation associated to the partition $\mathcal P_n$ (Theorem \ref{TheoremRepresentation}). Notice that if $R\neq id$, then $Q$ maps points from the backward $T$-orbit of $x$ onto the forward $T$-orbit. So $Q$ must be an $n$-permutation. The following result is an immediate corollary of Theorem \ref{TheoremRepresentation}.

\begin{proposition} For any point $x\in X$, the group $[[T]]_x$ is a maximal locally-finite subgroup of $[[T]]$.
\end{proposition}
\begin{proof} Consider any element $Q\in [[T]]\setminus [[T]]_x$. Using Theorem \ref{TheoremRepresentation}, we can write down $Q = PR$, where $P\in [[T]]_x$ and $R$ is a rotation. Since $Q\notin [[T]]_x$, we get that $R\neq id$. Notice that the group generated by $[[T]]_x$ and the element $Q$ contains $R$, which is an element of infinite order. This establishes the result.
\end{proof}

 Recall that  $\Gamma_x = [[T]]_x\cap [[T]]'$, where $x\in X$.  Each group $\Gamma_x$ is locally finite. The proof of Theorem \ref{TheoremProductLocallyFiniteGroups} is a modification of that of Lemma 4.1 from \cite{Matui:2006}.  We notice that Lemma 4.1 in \cite{Matui:2006} is heavily based on \cite[Remark 5.6]{GiordanoPutnamSkau:1999}, whose proof, in its turn, uses  operator-algebra techniques.  In the following lemma we give a purely combinatorial proof of \cite[Remark 5.6]{GiordanoPutnamSkau:1999}.  For a point $x\in X$, set $O^+(x) = \{T^n(x) : n\geq 0\}$ and $O^-(x) = \{T^n(x) : n < 0\}$. These are the {\it forward} and {\it backward} $T$-orbits of $x$, respectively. For each $Q\in [[T]]$, set $$a(Q) = card\left(Q(O^-(x))\cap O^{+}(x)\right)\mbox{ and }b(Q) = card\left(Q^{-1}(O^+(x))\cap O^{-}(x)\right).$$  Fix a $T$-invariant measure $\mu$ and set $\varphi(Q) = \int_X f_Q d\mu$. Notice that $\varphi :[[T]] \rightarrow \mathbb R$ is a non-trivial group homomorphism, see \cite{GiordanoPutnamSkau:1999}.

 \begin{lemma} For any $Q\in [[T]]$, we have that $\varphi(Q) = a(Q) - b(Q)$. In particular, $\varphi([[T]]) = \mathbb Z$ and the definition of $\varphi$ is independent of the choice of  $\mu$.
 \end{lemma}
 \begin{proof}  Let $P$ be an $n$-permutation. Then $\varphi(P) = 0$. Notice that $\varphi(T) = 1$ and $\varphi(T^{-1}) = -1$. If $T_C$ is the induced transformation on a clopen set $C$, then $T = T_C P$ for some periodic transformation $P$, see, for example, \cite[Proposition 2.1]{BezuglyiMedynets:2008}. Hence $\varphi(T_C) = 1$ and $\varphi(T_C^{-1}) = -1$.

 Consider an arbitrary $Q\in [[T]]$.  Using Theorem \ref{TheoremRepresentation}, find   an $n$-permutation $P$ and  an $n$-rotation $R$ with  $Q = PR$ and $R$ having the rotation number at most one. Then $\varphi(Q) = \varphi(R)$. Writing down $R$ as a product of induced transformations,
$$ R  = \prod_{a\in A}T_{ U(a)} \times \prod_{ b \in B} T_{D( b)}^{-1} $$ for some subsets  $A,B\subset \{0,1\ldots,m_n\}$, we get that $\varphi(Q) = card(A) - card(B)$. Since the set $A$ (the set $B$) represents the levels that are mapped by $Q$ over the top (bottom) of a partition, we get that  $card(A) = a(Q)$ and $card(B) = b(Q)$.
 \end{proof}

\begin{theorem}\label{TheoremProductLocallyFiniteGroups} Let $(X,T)$ be a Cantor minimal system. Then  $[[T]]' = \Gamma_x\cdot \Gamma_y$, where $x$ and $y$ are points from distinct $T$-orbits.
\end{theorem}
\begin{proof}

Let $Q\in \Gamma$. Let $\{\Xi_n\}_{n\geq 1}$ be a Kakutani-Rokhlin partition built over the point $x$.  Applying Theorem \ref{TheoremRepresentation} to the element $Q$ and partitions $\{\Xi_n\}$, we can represent $Q$  as $Q = P_nR_n$, where $P_n$ is an $n$-permutation and $R_n$ is an $n$-rotation (with respect to these partitions).  Clearly, $P_n\in [[T]]_x$.

Set $$I^+ = \{n\geq 0: Q(T^n(x))\in O^-(x)\}$$ and
 $$I^- = \{n < 0: Q(T^{n}(x))\in O^+(x)\}.$$

It follows from the lemma above that $card(I^+) = card(I^-) = m$ for some $m$.  Using the minimality of $T$, we can find a clopen set $Y$ of the form $Y= C\sqcup T^p(C)$ (for some $p\in \mathbb Z$) such that $x\in Y$ and $T^i(Y)\cap T^j(Y) = \emptyset$ for any $-q\leq i\neq j\leq q$, where $q = \max_{x\in X}|f_Q(x)|$, and such that none of the sets $\{T^j(Y) : -q\leq j \leq q\}$ contains the point $y$.

Fix any bijection $\pi$  that maps $I^-$ onto $I^+$ and $I^+$ onto $I^-$. Set $I = I^-\cup I^+$.  Define a homeomorphism $P_2$
as follows: $P_2| T^n(Y) = T^{\pi(n)-n} |T^n(Y)$ for any $n\in I$ and $P_2 = id$ elsewhere. Observe that $P_2\in [[T]]$ and $P_2$ is an involution.
 We notice that $P_2\in [[T]]_y$ as the set $Y$ was chosen so that $\{T^j(Y) : -q\leq j \leq q\}$ did not contain the point $\{y\}$.
  Since $Y = C\sqcup T^p(C)$, the element $P_2$ is a product of two isomorphic involutions (built over $C$ and $T^p(C)$, respectively). This implies that $P_2\in [[T]]'$. Hence $P_2\in [[T]]_y$.

  Set $P_1 = QP_2^{-1}\in [[T]]'$. Let  $x_n = T^n(x)$.  If $n\in I^+$, then $P_2^{-1} x_n$ is moved by $Q$ to the forward orbit of $x$. Thus, $P_1x_n = Q P_2^{-1} x_n\in O^+(x)$. Similarly, if $n\notin I^+$, then $P_2(x_n)\in O^+(x)$. This implies that  $P_1\in [[T]]_x$.
\end{proof}

%
%

\end{document}